\documentclass[12pt,letterpaper]{article}
\usepackage{graphicx}
\usepackage{amsmath, amssymb, amsthm, mathrsfs,verbatim,tensor}
\usepackage[left=1in,top=1in,right=1in]{geometry}
\usepackage{setspace}

\newtheorem*{claim}{Claim}

\newtheorem{theorem}{Theorem}
\newtheorem{proposition}{Proposition}

\newcommand{\R}{\ensuremath{\mathbb{R}} }

\newcommand{\Z}{\ensuremath{\mathbb{Z}} }

\newcommand{\F}{\ensuremath{\mathcal{F}} }

\newcommand{\cf}{\mathbf{1}}
\renewcommand{\P}{\mathbb{P}}

\newcommand{\B}{\mathcal{B}}

\title{Strange uniform random variables}
\author{Douglas Rizzolo\footnote{This work was supported in part by NSF grant DMS-1204840 and in part by the National Science Foundation Graduate  Research Fellowship under Grant No. DGE 1106400.} \\ Department of Mathematics \\ University of Washington}
\date{\today}

\begin{document}
\maketitle

\begin{abstract}
In probability theory, there is a tendency to treat one random variable with a given distribution as being just as good as any other.  By and large this is fine because probability is (mostly) concerned with distributional properties of random variables.  However, every now and again we are forced to deal with non-distributional properties.  In this paper we investigate how different random variables with the same distribution can be.  Specifically, we construct random variables that are all uniformly distributed on the unit interval, but that nonetheless have strikingly different properties.
\end{abstract}

In probability theory, there is a tendency to treat one random variable with a given distribution as being just as good as any other.  By and large this is fine because probability is (mostly) concerned with distributional properties of random variables.  However, every now and again we are forced to deal with non-distributional properties.  For example, when working with L\'evy processes, one often works with a c\'adl\'ag version of the process even though the property of having c\'adl\'ag paths is not distributional.  Or, when working with Markov processes, one often works with the canonical version of the process on path space even though this is only one particular instance of the process.  These examples highlight a natural question: given two random variables with the same distribution, how different can they really be?

In this note, we will mostly restrict ourselves to considering random variables that are uniformly distributed on $[0,1]$.  Throughout this note we let $\B[0,1]$ be the Borel $\Omega$-algebra on $[0,1]$ and let $\lambda$ be Lebesgue measure on $[0,1]$ restricted to Borel sets.  We will let $(\Omega, \P, \F)$ denote an arbitrary, but not fixed, probability space.  Recall that a measurable function $X:\Omega \to [0,1]$ is called uniformly distributed if for all $A\in \B[0,1]$ we have $ \P(X^{-1}(A)) = \lambda(A)$.  From this point on, we will adopt the usual probabilistic convention of using $\P(X\in A)$ instead of $ \P(X^{-1}(A))$.

Let us begin our investigation by looking at the following scenario.  Suppose that $X$ and $Y$ are uniformly distributed random variables (not necessarily defined on the same basic probability space) and that $h:[0,1]\to \R$ is a function such that $h(X)$ and $h(Y)$ are Borel measurable.  Does it follow that $h(X)=_d h(Y)$; that is, are $h(X)$ and $h(Y)$ necessarily equal in distribution?  When $h$ itself is measurable the answer is clearly yes, but the case when $h$ is nonmeasurable is less clear.  We will show that the answer is no, $h(X)$ and $h(Y)$ need not be equal in distribution.  For the sake of full disclosure, we remark that this answer depends on using the Axiom of Choice.  Later on we will comment on what can happen if the Axiom of Choice is not assumed. 

Our first result is the following theorem.

\begin{theorem}\label{theorem 1}
Let $\mu$ and $\nu$ be two probability measures on $\R$.  There exist uniformly distributed random variables $X:\Omega_1\to [0,1]$ and $Y:\Omega_2 \to [0,1]$ and a function $h:[0,1]\to \R$ such that $h(X)$ is measurable and has law $\mu$ and $h(Y)$ is measurable and has law $\nu$.
\end{theorem} 

In fact, we can do even better (or worse, depending on your perspective).

\begin{theorem}\label{theorem 2}
Let $I$ be an index set with cardinality at most that of the continuum and let $\{\mu_\alpha\}_{\alpha\in I}$ be a collection of probability measures on $\R$.  There exists a family $\{X_\alpha\}_{\alpha\in I}$ of uniformly distributed random variables and a (single) function $h:[0,1]\to \R$ such that, for all $\alpha$, $h(X_\alpha)$ is measurable and has law $\mu_\alpha$.
\end{theorem}

Another way to look at the difference between two random variables is to look at the set of values they take.  More precisely, suppose that $X:\Omega_1\to [0,1]$ and $Y:\Omega_2\to [0,1]$ are uniformly distributed.  What can we say about $X(\Omega_1)\cap Y(\Omega_2)$?  Since $\lambda(\{x\})=0$ for all $x\in [0,1]$ there is no particular $x\in [0,1]$ that must be in this intersection, but how small can it really be?  Could it be empty?  This last question can be phrased more dramatically: is there a uniform random variable $X$ such that we can remove all of the values it obtains from $[0,1]$ and then build another uniform random variable that only takes values in what is left?  Perhaps surprisingly, the answer is yes.  We will prove the following strong version of the classical fact that a uniform random variable does not take any particular value.

\begin{theorem} \label{theorem 3}
There exist uniformly distributed random variables $X:\Omega_1\to [0,1]$ and $Y:\Omega_2 \to [0,1]$ such that $X(\Omega_1)\cap Y(\Omega_2) = \emptyset$.
\end{theorem}

Again, we can prove a stronger version.

\begin{theorem}\label{theorem 4}
Let $I$ be an index set with cardinality at most that of the continuum. There exists a family $\{X_\alpha\}_{\alpha\in I}$ of uniformly distributed random variables defined on probability spaces $\{\Omega_\alpha\}_{\alpha\in I}$ such that $X_\alpha(\Omega_\alpha)\cap X_\beta(\Omega_\beta) = \emptyset$ whenever $\alpha\neq \beta$.
\end{theorem}

We remark that Theorem \ref{theorem 4} is quite similar to \cite[Example 7.30]{WiHa93}, which proves essentially the same result using normally distributed random variables instead of uniformly distributed ones.  Both results are essentially probabilistic reformulations of a classical theorem from measure theory that we will come to in due course (for the impatient reader, it is Theorem \ref{theorem 6} below). 

Let $\lambda_*$ and $\lambda^*$ be Lebesgue inner and outer measure respectively.  We can also prove the following theorem.

\begin{theorem}\label{theorem 5}
There exists a family $\{X_\alpha\}_{\alpha\in [0,1]}$ of uniformly distributed random variables defined on probability spaces $\{\Omega_\alpha\}_{\alpha\in [0,1]}$ such that whenever $\alpha\neq \beta$ we have
\[ \lambda_*(X_\alpha(\Omega_\alpha)\cap X_\beta(\Omega_\beta)) = \min(\alpha,\beta) \quad \textrm{and} \quad \lambda^*(X_\alpha(\Omega_\alpha)\cap X_\beta(\Omega_\beta)) = \max(\alpha,\beta).\]
\end{theorem}

Theorems \ref{theorem 1} and \ref{theorem 2} can be obtained as consequences of Theorems \ref{theorem 3} and \ref{theorem 4} respectively. 

\begin{proof}[Proof of Theorem \ref{theorem 2}]
Let $\{X_\alpha\}_{\alpha\in I}$ of uniformly distributed random variables defined on probability spaces $\{\Omega_\alpha\}_{\alpha\in I}$ as in Theorem \ref{theorem 4}.  For $\alpha\in I$, let $F_\alpha$ be the cumulative distribution function for $\mu_\alpha$ and let $F^{-1}_\alpha$ be its right-continuous inverse.  Define $h:[0,1]\to \R$ by
\[ h(x) = \sum_{\alpha \in I} F_{\alpha}^{-1}(x) \cf(x \in X_\alpha(\Omega_\alpha)).\]
Note that for each $x\in [0,1]$ at most one term in the sum is non-zero, so there is no need to worry about the fact that it is summing over a large index set.
\end{proof}

It remains to prove Theorem \ref{theorem 4}.  The first step is to figure out exactly which subsets of $[0,1]$ can be the range of a uniform random variable.  Suppose that $X:\Omega \to [0,1]$ is uniformly distributed and consider $X(\Omega)$.  Observe that $X(\Omega)$ must intersect every set of positive measure since for all $A\in \B[0,1]$ we have $\P(X\in A) = \P[X\in (A\cap X(\Omega))]$.  Thus $\lambda^*(X(\Omega)) = 1$.  In fact, this condition is also sufficient for a subset of $[0,1]$ to be the range of a uniform random variable.  Precisely,

\begin{proposition}\label{proposition subsets}
A subset of $[0,1]$ is the range of some uniform random variable if and only if it has outer measure $1$.
\end{proposition}

\begin{proof}
From our discussion above, it remains to show that if $A\subseteq [0,1]$ has outer measure $1$ then there is a uniform random variable with range $A$.  We will construct such a random variable explicitly.  Let $\B_A[0,1] = \{C\cap A : C\in \B[0,1]\}$.  It is not hard to check that $\B_A$ is a $\sigma$-algebra on $A$.  

\begin{claim}
If $B,C\in \B[0,1]$ and $B\cap A = C\cap A$ then $\lambda(B)=\lambda(C)$.
\end{claim}

\begin{proof}[Proof of claim]
In this case we have $A\cap B\cap C = A\cap B = A\cap C$, so that $B\setminus (B\cap C) \subseteq [0,1]\setminus A$.  Since $\lambda^*(A)=1$ this implies that $\lambda(B \setminus (B\cap C)) =0$, so that $\lambda(B) = \lambda(B\cap C)$.  Similarly, $\lambda(C) = \lambda(B\cap C)$, which proves the claim.
\end{proof}

We define $\lambda_A : \B_A[0,1] \to [0,1]$ by $\lambda_A(C\cap A) = \lambda(C)$.  The claim shows that $\lambda_A$ is well defined.  It is straightforward to check that $\lambda_A$ is a probability measure on $(A,\B_A[0,1])$.  Define $X:(A,\lambda_A,\B_A[0,1]) \to ([0,1],\B[0,1])$ by $X(x)=x$.  It follows immediately that $X$ is uniformly distributed and has range equal to $A$.
\end{proof}

Proposition \ref{proposition subsets} reduces the proof of Theorem \ref{theorem 4} to finding continuumly many pairwise disjoint subsets of $[0,1]$ each of which has outer measure $1$.  Fortunately for us, this is a classical result from measure theory due to Lusin and Sierpinksi,

\begin{theorem}[Lusin and Sierpinski \cite{SiLu17}] \label{theorem 6}
There exists a partition of $[0,1]$ into continuumly many pairwise disjoint sets, each of which has outer measure $1$.
\end{theorem}

The original reference for this Theorem is \cite{SiLu17}, and this paper has recently been reprinted in \cite[p. 26]{Sina03}.  See also \cite{Abia76} for a slightly more general version where $[0,1]$ is replaced by an arbitrary closed set and the condition of having outer measure $1$ is replaced by having full outer measure.  

Something like the Axiom of Choice is needed to prove Theorem \ref{theorem 6}.  To see this, let $(A_\alpha)_{\alpha \in I}$ be a partition of $[0,1]$ into continuumly many sets such that $\lambda^*(A_\alpha)=1$ for all $\alpha\in I$.  Observe that for each $\alpha$, $A_\alpha$ must be non-Lebesgue-measurable since we also have $\lambda^*([0,1]\setminus A_\alpha)=1$.  The existence of non-measurable sets is connected, to some degree, with the Axiom of Choice.  For example, there exists a model of the Zermelo-Frankel axioms of set theory (without Choice) in which the set of real numbers is a countable union of countable sets (see \cite[Theorem 10.6]{Jech73}).  In such a model, all sets are Borel, so nothing like what we have been discussing here can happen.  Of course, it is also near impossible to do any analysis in such a model.  Even if we allow enough Choice to do most analysis, things can go wrong.  For instance (under some additional assumptions) there models of the Zermelo-Frankel axioms with the Axiom of Dependent Choice in which every set of reals is Lebesgue measurable (see \cite{Shel84, Solo70} for details).   

\begin{proof}[Proofs of Theorems \ref{theorem 4} and \ref{theorem 5}]
For Theorem \ref{theorem 4}, let $(A_\alpha)_{\alpha \in I}$ be a partition of $[0,1]$ as in Theorem \ref{theorem 6} and for each $\alpha\in I$ let $X_\alpha$ be a uniform random variable with range $A_\alpha$, as guaranteed to exist by Proposition \ref{proposition subsets}.  For Theorem \ref{theorem 5}, let $(A_\alpha)_{\alpha \in [0,1]}$ be a partition of $[0,1]$ as in Theorem \ref{theorem 6}.  Define $B_\alpha = A_\alpha \cup [0,\alpha]$ and each $\alpha\in [0,1]$ let $X_\alpha$ be a uniform random variable with range $B_\alpha$, again as guaranteed to exist by Proposition \ref{proposition subsets}.
\end{proof}

For the proof of Theorem \ref{theorem 6} we refer the reader to the references mentioned above.  We content ourselves to sketch the proof of the following weaker theorem,

\begin{theorem}\label{theorem 7}
There exists a partition of $[0,1]$ into two sets, each of which has outer measure $1$.
\end{theorem}

For us, the advantage of proving Theorem \ref{theorem 7} instead of Theorem \ref{theorem 6} is that we can prove it using a minor variation of the standard construction of a non-measurable set.  We follow the approach in \cite[Section 16]{Halm74}, to which we refer for a more detailed version of the proof.

\begin{proof}
Let $\zeta$ be an irrational number and let $B = \{n+m\zeta : n,m\in \Z\}$.  Further define $B_e=\{n+m\zeta \in A : n \textrm{ is even}\}$ and $B_o= B\setminus B_e$.  It is not difficult to see that $B$, $B_e$, and $B_o$ are each dense subsets of $\R$.  We define an equivalence relation on $\R$ by $x\sim y$ if and only if $x-y\in B$.  Using the Axiom of Choice, we let $C\subseteq \R$ be a set containing exactly one representative of each equivalence class in $\R/\sim$.  Let $M=C+B_e$.  Note that the difference set $D(M):= \{m_1-m_2 : m_1,m_2\in M\}$ is disjoint from $B_o$.  However, the difference set of any measurable set with positive Lebesgue measue contains an open interval about the origin.  Since $B_o$ is dense, it follows that $M$ has no subsets of positive measure.  Furthermore, we see that $\R\setminus M = M+1$, so the complement of $M$ also contains no subset of positive measure.  It follows that if we define $A_1=M\cap [0,1]$ and $A_2= [0,1]\setminus M$ then $\{A_1,A_2\}$ is a partition of $[0,1]$ into two sets, each of which has outer measure $1$.
\end{proof}

\bibliographystyle{plain}
\bibliography{uniformbib}

\end{document}